\documentclass{article}

\usepackage[english]{babel}
\usepackage[utf8x]{inputenc}
\usepackage[T1]{fontenc}

\usepackage[letterpaper,margin=1.5in,marginparwidth=1in]{geometry}

\usepackage{amsmath}
\usepackage{amssymb}
\usepackage{amsthm}
\usepackage{graphicx}
\usepackage{enumerate}
\usepackage{tikz}
\usepackage{stmaryrd}
\usepackage[colorinlistoftodos]{todonotes}
\usepackage[colorlinks=true, allcolors=blue]{hyperref}

\title{Metrizability of Mahavier products indexed by partial orders}
\author{Steven Clontz and Jacob Dunham}

\newcommand{\term}{\textit}

\newcommand{\tuple}[1]{\left\langle{#1}\right\rangle}

\newcommand{\maProd}[1]{{\mathbf{M}}\tuple{#1}}

\newcommand{\concat}{^{\frown}}
\newcommand{\finSubseteq}{\subseteq_{\mathrm{finite}}}
\newcommand{\lof}[1]{{^{<}#1}}
\newcommand{\rof}[1]{{#1^{<}}}
\newcommand{\leof}[1]{{^{\leq}#1}}
\newcommand{\reof}[1]{{#1^{\leq}}}

\newcommand{\nleof}[1]{{^{\not\leq}#1}}
\newcommand{\nreof}[1]{{#1^{\not\leq}}}
\newcommand{\lofs}[2]{{^{#2}#1}}
\newcommand{\rofs}[2]{{#1^{#2}}}

      \theoremstyle{plain}
      \newtheorem{theorem}{Theorem}
      \newtheorem{lemma}[theorem]{Lemma}
      \newtheorem{corollary}[theorem]{Corollary}
      \newtheorem{proposition}[theorem]{Proposition}
      
      \newtheorem{question}[theorem]{Question}

      \theoremstyle{definition}
      \newtheorem{definition}[theorem]{Definition}
      \newtheorem{notation}[theorem]{Notation}
      \newtheorem{example}[theorem]{Example}
      
      \newtheorem{remark}[theorem]{Remark}

\begin{document}
\maketitle

\begin{abstract}
Let \(X\) be separable metrizable, and let
\(f\subseteq X^2\) be a non-trivial relation on \(X\).
For a given partial order \(\tuple{P,\leq}\),
the Mahavier product \(\maProd{X,f,P}\subseteq X^P\)
(also known as a generalized inverse limit)
collects functions such that \(x(p)\in f(x(q))\) for all
\(p\leq q\). Clontz and Varagona previously showed
for well orders \(P\)
that \(\maProd{X,f,P}\) is separable metrizable
exactly when \(P\) is countable and \(f\) satisfies condition \(\Gamma\); we extend this
result to hold for all partial orders.
\end{abstract}

\section{Introduction}

Let \(X\) be a separable metrizable topological space, let
\(f\subseteq X^2\) be a relation on \(X\),
and let \(Q\) be a set preordered by (reflexive and
transitive) \(\preceq\).

Extending work done in e.g. \cite{M5,Char}
we consider the subspace
\(\maProd{X,f,Q}\) of \(X^Q\) where
\(x(p)\in f(x(q))\) for all \(p\preceq q\).
Such subspaces are often known as generalized
inverse limits, or as we will refer to them,
Mahavier products.
For an introduction to such structures,
often considered in the context of continuum
theory, we direct the reader to \cite{MR3014043}.
Much of the literature on this subject
considers only simple indices,
particularly \(Q=\mathbb N\) or \(\mathbb Z\)
with its usual order.

It is immediate that any subspace of \(X^Q\)
is separable metrizable whenever \(Q\) is countable.
In \cite{CLONTZVARAGONA,GitC} the authors
consider whether \(\maProd{X,f,Q}\) might be metrizable
when \(Q\) is an uncountable well-order. It turns
out that, except in trivial situations, the answer
is no.

We aim to extend this result to the more general
case where the index set of the Mahavier product is any partial order, that is,
any antisymmetric preorder. To do this,
we utilize the notion of a 
partially-ordered topological
space (POTS) originally
defined in \cite{1954}, a generalization
of the class of linearly ordered
topological spaces (LOTS) studied broadly
throughout general topology.

\section{Partially-Ordered Topological
Spaces}

For convenience, we formally define the notion
of a preorder and partial order here.

\begin{definition}
A \term{preorder} \(\tuple{Q,\preceq}\) is a
set paired with a
reflexive (\(x\preceq x\)) and transitive
(\(x\preceq y\preceq z\Rightarrow x\preceq z\))
relation.
\end{definition}

We adopt the convention of using \(Q\) for
preorders as they are sometimes known as
\term{quasi-orders}.

\begin{definition}
A \term{partial order} \(\tuple{P,\leq}\) is
a preorder that is antisymmetric
(\(x<y\Rightarrow y\not<x\)).
\end{definition}

We note that every preorder admits a natural
partial order.

\begin{proposition}
Let \(Q\) be preordered by \(\preceq\), and set
\(p\sim q\) if and only if \(p\preceq q\) and
\(q\preceq p\). Then \(\sim\) is an equivalence
relation, and its set of equivalence classes
\(P=\{[p]:p\in Q\}\) ordered by
\(A\leq B\) if and only if \(p\preceq q\)
for all \(p\in A,q\in B\)
is a partial order.
\end{proposition}

\begin{notation}
For each preorder \(\tuple{Q,\preceq}\) we denote the partial
order given in the previous proposition by
\(\tuple{Q^\star,\leq}\).
\end{notation}

As we will see, for our purposes
we can use the partial order \(Q^\star\)
in place of any preorder \(Q\), so we now only consider
partial orders.

\begin{notation}
Let \(\tuple{P,\leq}\) be a partial order.
Then we adopt the following notation similar to
that commonly used for linear orders; for example:
\[(\leftarrow,p)=\lof p=\{r\in P:r<p\}\]
\[[p,\rightarrow)=\reof p=\{r\in P:r\geq p\}\]
\[(p,q]=\rof p\cap \leof q\]
\[\nreof p = P\setminus \reof p\]
\end{notation}

Of course, if a partial order is \term{total},
that is, \(p\leq q\) or \(q\leq p\) for all \(p,q\),
then we have the usual idea of a \term{total order}
or \term{linear order}, and e.g.
\(\nleof p=\rof p\). A wide class of topological
spaces known as \term{LOTS}
are defined in terms of linear orders; we
characterize them for all partial orders as follows.

\begin{definition}
A \term{partially ordered topological space}, or \term{POTS},
is a topological space \(X\) partially ordered by
\(\leq\) with a subbasis
\(\{\nleof x :x\in X\}\cup\{\nreof x:x\in X\}\).
\end{definition}

The reader may verify that if \(\leq\) is a linear order,
then this subbasis yields the usual
basis of open intervals
\(\{(x,y):x,y\in X\cup\{\leftarrow,\rightarrow\}\}\).
However,
the next example illustrates that a POTS may not even be
Hausdorff (though it will always be \(T_1\)).

\begin{example}
Let
\([\omega]^{<\aleph_0}=\{F\subseteq\omega:F\text{ is finite}\}\)
be partially ordered by \(\subseteq\).

We then let \(F\subseteq \omega\) be finite, and thus for any \(x\not\in F\), 
the singleton set \(\{x\}\in \lofs{F}{\not\subseteq}\).
Consider now the subbasic 
open set \(\rofs{F}{\not\subseteq}\). If \(F = \emptyset\) then
\(\rofs{F}{\not\subseteq} = \emptyset\) as well, and if \(F = \{y\}\)
then \(\{x\}\in \rofs{F}{\not\subseteq}\)
for all \(x\in \omega\setminus \{y\}\). 
Finally, if \(|F| > 1\), then \(\{x\}\in \rofs{F}{\not\subseteq}\) for all 
\(x\in \omega\). Thus, any two non-empty subbasic open sets have 
infinite intersection, showing 
\(\tuple{[\omega]^{<\aleph_{0}},\subseteq}\) cannot be Hausdorff.
\end{example}

To address this, Ward \cite{1954} gives the following definition
for ``continuous'' partial orders; some authors
\cite{gierz_hofmann_keimel_lawson_mislove_scott_2003} require
this as a condition of being a ``pospace.''

\begin{definition}
A POTS is said to be \term{continuous}
if for each \(p\not\leq q\) there exist open sets
\(U,V\) such that \(p\in U,q\in V\) and
\(r\not\leq s\) for all \(r\in U,s\in V\).
\end{definition}

Note that
the definition immediately implies the Hausdorff property,
and may be recharacterized as follows.

\begin{definition}
A subset \(A\) of a partial order \(\leq\) is said to be \term{downward}
if for all \(p\in A\), \(\leof p\subseteq A\).
A subset \(A\) of a partial order is said to be \term{upward}
if for all \(p\in A\), \(\reof p \subseteq A\).
\end{definition}

\begin{lemma}[\cite{1954}]\label{contPots}
The following are equivalent for a given POTS
\(\tuple{P,\leq}\).
\begin{itemize}
\item \(P\) is continuous.
\item For each \(p\not\leq q\), there exists an upward open
neighborhood \(U\) of \(p\) and a downward open neighborhood \(V\)
of \(q\) such that \(U,V\) are disjoint.
\end{itemize}
\end{lemma}

\section{Mahavier Products with Partial Orders}

We may define the Mahavier product as a subspace
of the usual Tychonoff product. By convention,
we will treat relations \(f\subseteq X^2\)
as set-valued functions, that is,
\(f(x)=\{y:\tuple{x,y}\in f\}\).

\begin{definition}
Let \(X\) be a topological space.
A relation \(f\subseteq X^2\) is said to be
a \term{\(V\)-relation} if it
is closed,
idempotent (\(f(f(x))=f^2(x)=f(x)\)),
surjective (\(\forall y\in X\exists x\in X(y\in f(x))\)),
and serial (\(\forall x\in X\exists y\in X(y\in f(x))\)).
\end{definition}

\begin{definition}
Let \(X\) be a topological space,
\(f\subseteq X^2\) be a \(V\)-relation,
and \(Q\) be a preorder. Then
\[\maProd{X,f,Q}=\{x\in X^Q:x(p)\in f(x(q))
\text{ for all }p\preceq q\}\]
is the \term{Mahavier product}, considered
as a subspace of the usual Tychonoff product
\(X^Q\).
\end{definition}

When \(Q=\mathbb N\) and \(f\) is singleton-valued,
then \(\maProd{X,f,Q}\)
is the standard topological inverse limit with
bonding map \(f\).
This definition may be also contrasted with other
interesting subspaces of \(X^Q\)
studied in general topology,
such as the \(\Sigma\)-product. (We pose
a question relating to compact subspaces of 
\(\Sigma\)-product, known as Corson compacta,
at the end of the paper.)

In \cite{GitC}, the first author and Varagona
showed the following.

\begin{definition}
A relation \(f\subseteq X^2\) satisfies
\term{condition \(\Gamma\)} if there exist distinct
\(x,y\in X\) such that
\(\tuple{x,x},\tuple{x,y},\tuple{y,y}\in f\).
\end{definition}

\begin{lemma}
Let \(X\) be weakly countably compact and
\(f\not=\iota=\{\tuple{x,x}:x\in X\}\)
(that is, \(f\) is \term{nontrivial}) be
a \(V\)-relation. Then \(f\) satisfies
condition \(\Gamma\).
\end{lemma}

To this end, we will concentrate on the following
particular Mahavier product.

\begin{definition}
Let \(X\) contain distinct elements \(0,1\).
Then \(\gamma_X\subseteq X^2\) satisfies
\(\gamma_X(0)=X\) and \(\gamma_X(x)=\{1\}\) otherwise.
\end{definition}

Of course, if \(X\) is \(T_1\) then \(\gamma_X\) is a \(V\)-relation 
that satisfies condition \(\Gamma\).
When \(X\) is assumed from context, we simply write \(\gamma\).

\begin{proposition}
Let \(Y\subseteq X\) and \(g\subseteq f\).
Then \(\maProd{Y,g,Q}\subseteq\maProd{X,f,Q}\).

In particular,
let \(2\subseteq X\) and \(f\) satisfy condition \(\Gamma\)
(that is, \(\gamma=\gamma_2\subseteq f\)).
Then \(\maProd{2,\gamma,Q}\subseteq\maProd{X,f,Q}\).
\end{proposition}

Recalling our motivation, we want to characterize
the separable metrizability of the Mahavier product \(\maProd{X,f,P}\)
(for separable metrizable \(X\) and 
nontrivial \(f\) satisfying condition \(\Gamma\)).
Since \(\maProd{X,f,P}\) is trivially separable and
metrizable when \(P\) is countable, we will consider
the case for uncountable \(P\). In particular,
we will show that the subspace \(\maProd{2,\gamma,P}\)
fails to be second-countable in this case.

First consider the following simplification for preorders
that allows us to concentrate on partial orders.
We use the symbol \(\cong\) to denote homeomorphic
spaces.

\begin{proposition}
Let \(Q\) be preordered by \(\preceq\).
Then \(\maProd{2,\gamma,Q}\cong\maProd{2,\gamma,Q^\star}\).
\end{proposition}
\begin{proof}
Note that for any \(x\in \maProd{2,\gamma,Q}\), if \(x(p) = 1\)
then \(x(q) = 1\) for all \(q\sim p\), due to \(q\preceq p\).
Likewise, if \(x(p)=0\), then \(x(q)=0\)
for all \(q\sim p\), due to \(p\preceq q\).
Thus
\(h:\maProd{2,\gamma,Q} \rightarrow \maProd{2,\gamma,Q^\star} \)
defined by \(h(x)([p]) = x(p)\) (where \([p]\in Q^\star\)
is the equivalence
class of \(p\in Q\)) is a well-defined homeomorphism.
\end{proof}

Many properties of \(\maProd{X,f,P}\)
are inherited from its superspace \(X^P\); for example,
it is \(T_3\) provided \(X\) is.
The following result shows that compactness is also preserved.

\begin{proposition}
Let \(f\subseteq X^2\) be closed. Then
\(\maProd{X,f,P}\) is a closed subspace of \(X^P\).
\end{proposition}

\begin{proof}
Let \(l\not\in\maProd{X,f,P}\). Then there
exist \(p<q\) such that \(l(p)\not\in f(l(q))\).
So choose open subsets \(U,V\) of \(X\) such that
\(\tuple{l(q),l(p)}\in U\times V\subseteq X^2\setminus f\).
Then \(\prod_{r\in P}U_r\) defined by \(U_p=V\), \(U_q=U\),
and \(U_r=X\) otherwise, is an open neighborhood of
\(l\) and disjoint from \(\maProd{X,f,P}\).
\end{proof}

\begin{corollary}
If \(X\) is compact and \(f\) is a \(V\)-relation, then
\(\maProd{X,f,P}\) is compact.
\end{corollary}

\section{Two illustrative examples}

The following is an example where \(X\)
is a compact metrizable space, 
\(P\) is a linear order, and
\(\maProd{X,\gamma,P}\) is a familiar
continuum (albeit non-metrizable,
as will follow later from
the fact that \(P\) is uncountable).

\begin{example}
Let \(X=[0,1]\) and \(P=[0,1]\) with the
usual (linear) order. Then
\(\maProd{X,\gamma,P}\) is a copy of
the lexicographic square.
\end{example}

\begin{proof}
The lexicographic square is given by
\(L=[0,1]^2\) where \(\tuple{x,y}<\tuple{w,z}\)
if and only if \(x<w\), or both \(x=w\)
and \(y<z\). This space is well-known to
be non-metrizable (this follows from the
fact that it is compact but non-separable). 

For \(f\in\maProd{X,\gamma,P}\)
let \(x_f=\inf\{x:f(x)=0\}\). It follows that
\(f(x)=1\) for \(x<x_f\) and
\(f(x)=0\) for \(x>x_f\), and \(f=g\) if and
only if \(x_f=x_g\) and \(f(x_f)=g(x_g)\).

So let \(\theta:\maProd{X,\gamma,P}\to L\)
be defined by \(\theta(f)=\tuple{x_f,f(x_f)}\).
It follows that \(\theta\) is a bijection.
We will show that \(\theta\) is a homeomorphism.

To see this, consider a nonempty subbasic open set
\(U=\prod_{x\in[0,1]}U_x\) containing \(f\), where
each \(U_x=[0,1]\) except for a single
\(x'\in X\). First assume \(U_{x'}=(y',1]\).
It follows that \(f(x)=1\) for all \(x<x'\).
We claim \(\theta[U]=(\tuple{x',y'},\rightarrow)\).

To see this, first note that 
\(\theta(f)>\tuple{x',y'}\) follows from
\(x_f\geq x'\), and if \(x_f=x'\) then
\(f(x_f)=f(x')>y'\). Therefore
\(\theta[U]\subseteq(\tuple{x',y'},\rightarrow)\).

Likewise if \(\tuple{a,b}>\tuple{x',y'}\),
note \(x'\leq a\) and \(x'=a\Rightarrow b>y'\), and
define \(g\) by
\[
g(x)=
\begin{cases}
1 & \text{if } x<a \\
b & \text{if } x=a \\
0 & \text{if } x>a \\
\end{cases}
.\]
Then \(x_g=a\), and
\(\theta(g)=\tuple{x_g,g(x_g)}=
\tuple{a,b}\). Finally
\(g\in U\) since \(x'<a\) implies \(g(x')=1\geq y'\)
and \(x'=a\) implies
\(g(x')=g(a)=b>y'\).
Thus
\(\theta[U]\supseteq(\tuple{x',y'},\rightarrow)\),
completing the proof of our claim.

This argument may be adapted to show
\(U_{x'}=[0,y')\) implies
\(\theta[U]=(\leftarrow,\tuple{x',y'})\).
Therefore, \(\theta\) is an open continuous
map, so \(\theta\) is a homeomorphism.
\end{proof}

The following example illustrates an interesting
space produced as the Mahavier
product of a non-linear partial order;
since \(P\) is countable, this is a
metrizable subspace of \([0,1]^P\).

\begin{example}
Let \(2^{<\omega}\) be the Cantor tree of
finite sequences of \(0\) and \(1\) ordered by
extension.
An illustration of a typical
element \(f\in\maProd{[0,1],\gamma,2^{<\omega}}\)
and its
neighborhood is given in Figure \ref{maCantor}.
Each filled dot represents \(s\in2^{<\omega}\),
where each \(0\) in the sequence moves northwest and each
\(1\) in the sequence moves northeast.
If the tree meets the dot for \(s\) then \(f(s)=1\); if the tree
meets \(s\) but not \(s\concat\tuple{n}\) then 
\(0\leq f(s\concat\tuple{n})<1\). A neighborhood allows for
\(\varepsilon\) error on each coordinate not determined by
\(\gamma\) (represented by the open dots and thinner lines).
\end{example}

\begin{figure}
\begin{center}
\begin{tikzpicture}
\draw[black, thick] (0,0)--(0,1);
\draw[black, thick] (0,1)--(1,2);
\draw[black, thick] (0,1)--(-1,2);
\draw[black, thick] (-1,2)--(-1.3,2.6);
\draw[blue] (-1.2,2.4) circle (1pt);
\draw[blue] (-1.4,2.8) circle (1pt);
\draw[blue] (-1.2,2.4)--(-1.4,2.8);
\draw[black, thick] (-1,2)--(-.6,2.8);
\draw[blue] (-.575,2.85) circle (1pt);
\draw[blue] (-.625,2.75) circle (1pt);
\draw[blue] (-.575,2.85)--(-.625,2.75);
\draw[black, thick] (1,2)--(1.5,3);
\draw[blue] (0.9,2.2) circle (1pt);
\draw[blue] (1,2) -- (0.9,2.2);
\draw[black, thick] (1.5,3)--(1.30,3.8);
\draw[blue] (1.3125,3.75) circle (1pt);
\draw[blue] (1.2825,3.85) circle (1pt);
\draw[blue] (1.325,3.7) -- (1.275,3.9);
\draw[black, thick] (1.5,3)--(1.65,3.6);
\draw[blue] (1.6,3.4) circle (1pt);
\draw[blue] (1.7,3.8) circle (1pt);
\draw[blue] (1.6,3.4) -- (1.7,3.8);
\filldraw[red] (-2pt,-2pt) rectangle (2pt,2pt);
\filldraw[red] (0,1) circle (2pt);
\node[anchor=west] at (0,1) {\tiny\(f\tuple{}=1\)};
\filldraw[red] (1,2) circle (2pt);
\node[anchor=west] at (1,2) {\tiny\(f\tuple{1}=1\)};
\filldraw[red] (-1,2) circle (2pt);
\node[anchor=east] at (-1,2) {\tiny\(f\tuple{0}=1\)};
\filldraw[red] (1.5,3) circle (2pt);
\node[anchor=west] at (1.5,3) {\tiny\(f\tuple{1,1}=1\)};
\filldraw[red] (.5,3) circle (2pt);
\node[anchor=south] at (.5,3) {\rotatebox{15}{\tiny\(f\tuple{1,0}=0\)}};
\filldraw[red] (-.5,3) circle (2pt);
\node[anchor=south] at (-.5,3) {\rotatebox{15}{\tiny\(f\tuple{0,1}=0.8\)}};
\filldraw[red] (-1.5,3) circle (2pt);
\node[anchor=east] at (-1.5,3) {\tiny\(f\tuple{0,0}=0.6\)};
\filldraw[red] (1.75,4) circle (2pt);
\node[anchor=south] at (1.75,4) {\tiny\(0.6\)};
\filldraw[red] (1.25,4) circle (2pt);
\node[anchor=south] at (1.25,4) {\tiny\(0.8\)};
\filldraw[red] (.75,4) circle (2pt);
\node[anchor=south] at (.75,4) {\tiny\(0\)};
\filldraw[red] (.25,4) circle (2pt);
\node[anchor=south] at (.25,4) {\tiny\(0\)};
\filldraw[red] (-.25,4) circle (2pt);
\node[anchor=south] at (-.25,4) {\tiny\(0\)};
\filldraw[red] (-.75,4) circle (2pt);
\node[anchor=south] at (-.75,4) {\tiny\(0\)};
\filldraw[red] (-1.75,4) circle (2pt);
\node[anchor=south] at (-1.75,4) {\tiny\(0\)};
\filldraw[red] (-1.25,4) circle (2pt);
\node[anchor=south] at (-1.25,4) {\tiny\(0\)};
\end{tikzpicture}
\end{center}
\caption{An element of
\(\maProd{[0,1],\gamma,2^{<\omega}}\)
with a typical neighborhood}\label{maCantor}
\end{figure}
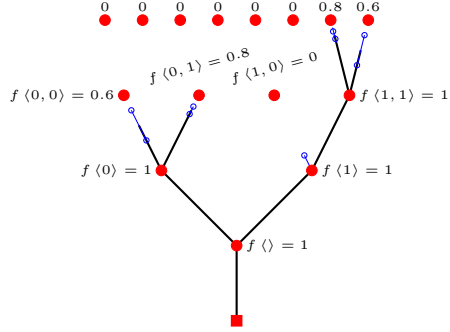

\section{Second-Countability and
the POTS of Downward Subsets}

In order to
demonstrate that \(\maProd{2,\gamma,P}\)
is second-countable exactly
when \(P\) is countable, we will show that it is homeomorphic
to a POTS derived from \(P\).

Recall the well-known compactification
\(\hat{L}=\{A\subseteq L:A\text{ is closed and downward}\}\)
of a given LOTS \(L\), that is, its Dedekind completion.
The requirement that \(A\) be
closed is necessary for \(\{\leof l : l\in L\}\cong L\)
to be dense in \(\hat{L}\). But by dropping
that requirement, we will obtain our desired copy
of \(\maProd{2,\gamma,P}\).

\begin{definition}
Let \(P\) be a partial order. Then we denote
the \term{POTS of downward subsets} by
\(\check{P}=\{A\subseteq P:A\text{ is downward}\}\)
partially ordered by the subset relation \(\subseteq\).
\end{definition}

\begin{lemma}
Let \(x\in P\), then \(\rofs{x}{\not\leq}\) is downward
and \(\lofs{x}{\not\leq}\) is upward.
\end{lemma}

\begin{proof}
To show \(\rofs{x}{\not\leq}\) is downward,
let \(x\in P\), \(b\in\rofs{x}{\not\leq}\) and \(a < b\).
We will show that \(a\in\rofs{x}{\not\leq}\),
that is, \(a\not\geq x\). If \(a\geq x\), then \(x \leq a< b\) implies
\(x\leq b\) by transitivity, contradicting the fact that
\(b\in\rofs{x}{\not\leq}\). The proof that \(\lofs{x}{\not\leq}\) is
upward is similar.
\end{proof}

\begin{proposition}
Let \(P\) be a partial order. Then
\(\tuple{\check P,\subseteq}\)
is a continuous POTS.
\end{proposition}
\begin{proof}
It is immediate that \(\subseteq\) is a partial order.
To see that it is continuous on \(\check{P}\), we will verify the last
bullet of Lemma \ref{contPots}. Let 
\(X,Y\in \check P\) with \(X\not\subseteq Y\) and \(x\in X\setminus Y\).
Note \(\leof{x},\rofs{x}{\not\leq}\in\check P\), and furthermore 
\(X\in \lofs{(\rofs{x}{\not\leq})}{\not\subseteq}\) since \(x\in X\),
and \(Y\in \rofs{(\leof x)}{\not\subseteq}\) since \(x\not\in Y\).

By the previous lemma these sets are upward and downward respectively.
So we now show these two open sets are disjoint. If
\(A\in\rofs{(\leof x)}{\not\subseteq}\cap
\lofs{(\rofs{x}{\not\leq})}{\not\subseteq}\), then \(A\) misses some
some \(a\leq x\) and contains some \(b\geq x\). But since \(b\geq x\geq a\),
this means \(A\) is not downward, a contradiction.
\end{proof}

\begin{remark}
There is a natural bijection between
elements of \(\maProd{2,\gamma,P}\) and \(\check{P}\).
Each sequence \(x\in\maProd{2,\gamma,P}\) satisfies
the property that \(x(q)=1\) implies \(x(p)=1\) for
all \(p\leq q\), showing that \(\{p\in P:x(p)=1\}\) is
downward. Likewise, the characteristic function
\(f(p)=1\) for \(p\in A\) and \(f(p)=0\) for \(p\not\in A\)
for each downward set \(A\in\check{P}\) belongs
to \(\maProd{2,\gamma,P}\).
\end{remark}

To see that this
bijection is actually a homeomorphism, we first introduce
a convenient basis for \(\check P\) as an alternative
to its POTS subbasis.

\begin{proposition}
Let
\(B\tuple{T,F}=\{A\in\check P:T\subseteq A\text{ and }F\cap A=\emptyset\}\). Then \(\{B\tuple{T,F}:T,F\finSubseteq P\}\)
is a basis for a topology on \(\check P\).
\end{proposition}

\begin{theorem}
This basis induces the POTS topology.
\end{theorem}
\begin{proof}

Consider \(T,F\finSubseteq P\). Let \(A\in B\tuple{T,F}\), so for any
\(t\in T\), \(A\in \lofs{(\nreof t)}{\not\subseteq}\), that is,
\(A\not\subseteq\nreof t\) since
\(t\in A\). Likewise,
for any \(f\in F\), \(A\in \rofs{(\leof f)}{\not\subseteq}\), that is,
\(A\not\supseteq\leof f\) since
\(f\not\in A\).
Therefore,
\[A\in \bigcap\left\{\lofs{(\nreof t)}{\not\subseteq}:t\in T\right\}
\cap \bigcap\left\{\rofs{(\leof f)}{\not\subseteq}:f\in F\right\}.\]
 
Now assume
\[A\in \bigcap\left\{\lofs{(\nreof t)}{\not\subseteq}:t\in T\right\}
\cap \bigcap\left\{\rofs{(\leof f)}{\not\subseteq}:f\in F\right\}.\]
Then it follows that (since \(A\) is downward),
\(t\in A\) for all \(t\in T\), 
and similarly \(f\not\in A\) for all \(f\in F\).
Thus, \(A\in B\tuple{T,F}\), showing that 
\[B\tuple{T,F}=
\bigcap\left\{\lofs{(\nreof t)}{\not\subseteq}:t\in T\right\}
\cap \bigcap\left\{\rofs{(\leof f)}{\not\subseteq}:f\in F\right\}\]
and therefore \(B\tuple{T,F}\) is open in the standard POTS topology.

Now consider \(A\in \lofs{R}{\not\subseteq}\) for some \(R\in \check{P}\),
so there exists some \(r\in A\) with \(r\not\in R\).
Therefore \(A\in B\tuple{\{r\},\emptyset}\subseteq R^{\not\subseteq}\).
Similarly, if \(A\in \rofs{R}{\not\subseteq}\), then there exists some
\(r\in R\) with \(r\not\in A\), showing that
\(A\in B\tuple{\emptyset,\{r\}}\subseteq \rofs{R}{\not\subseteq}\).
Therefore \(\lofs{R}{\not\subseteq}\) and \(\rofs{R}{\not\subseteq}\)
are open in the topology generated by the sets \(B\tuple{T,F}\).
\end{proof}

\begin{corollary}
\(\check P \cong \maProd{2,\gamma,P}\).
\end{corollary}
\begin{proof}
The map \(h:\check P\to\maProd{2,\gamma,P}\)
defined by
\[h(A)(p)=\begin{cases}0&p\not\in A\\1&p\in A\end{cases}\]
was earlier shown to be a bijection.
We now consider \(h[B\tuple{T,F}]\). As \(T\subseteq A\) and
\(F\cap A = \emptyset\) for all \(A\in B\tuple{T,F}\),
\(h\) will map the set \(B\tuple{T,F}\) to the set of sequences which are
restricted to \(\{1\}\) on \(T\), and \(\{0\}\) on \(F\).

Thus, let \(U_p = \{1\}\) if \(p\in T\), \(U_p = \{0\}\) if \(p\in F\),
and \(U_p = 2\) otherwise.
Then \(h[B\tuple{T,F}] = \prod_{p\in P}U_p\),
showing \(h\) is an open map. This also shows that \(h\) is continuous;
given nonempty \(\prod_{p\in P}U_p\), let \(T=\{p\in P:U_p=\{1\}\}\)
and \(F=\{p\in P:U_p=\{0\}\}\); then \(h[B\tuple{T,F}]=\prod_{p\in P}U_p\).
\end{proof}

Due to this homeomorphism,
the following theorem completes our
characterization of separable/compact metrizable
\(\maProd{X,f,P}\).

\begin{theorem}
\(\check P\) is second-countable if and only if
\(P\) is countable.
\end{theorem}
\begin{proof}
If \(P\) is countable, \(2^P\) is second-countable,
and therefore its subspace
\(\maProd{2,\gamma,P}\cong\check P\) is second-countable.

Suppose \(P\) is uncountable and \(\mathcal{B}\) is a basis for 
\(\check{P}\). Then for all \(p\in P\), fix \(B_p\in \mathcal{B}\) such 
that \(\leof p\in B_p\subseteq B\tuple{\{p\},\emptyset}\). Now let 
\(p\neq q\), and without loss of generality assume \(p\not\leq q\). 
Therefore, \(\leof q\not\in B\tuple{\{p\},\emptyset}\), and so 
\(q^{\leq}\not\in B_p\). By construction we have \(q^{\leq} \in B_q\), and
so it must be that \(B_p\neq B_q\). Thus, each \(B_p\) is unique, and so 
\(\mathcal{B}\) must be uncountable, showing that \(\check{P}\) is not second-countable.
\end{proof}

\begin{corollary}\label{sepMet}
Let \(X\) be separable metrizable, \(f\) satisfy
condition \(\Gamma\), and \(P\) be a partial order.
Then \(\maProd{X,f,P}\) is
separable metrizable if and only if \(P\) is countable.
\end{corollary}
\begin{proof}
If \(P\) is countable, \(X^P\) is separable metrizable,
and therefore its subspace
\(\maProd{X,f,P}\) is separable metrizable.

If \(P\) is uncountable, \(\maProd{X,f,P}\) contains
a non-second-countable subspace \(\check P\),
and therefore cannot be second-countable itself.
\end{proof}

Note that Corollary \ref{sepMet} is a generalization
of \cite[Theorem 5.1]{GitC}, as \(\check\alpha=\alpha+1\)
for all ordinals \(\alpha\). We also have the following.

\begin{corollary}
Let \(X\) be compact metrizable, \(f\) satisfy
condition \(\Gamma\), and \(P\) be a partial order.
Then \(\maProd{X,f,P}\) is
compact metrizable if and only if \(P\) is countable.
\end{corollary}

Of course these corollaries cannot extend to preorders
\(\tuple{Q,\preceq}\): if \(p\preceq q\) and \(q\preceq p\)
for all \(p,q\in Q\), then \(\maProd{2,\gamma,Q}\cong 2\)
regardless of the cardinality of \(Q\).




In \cite{GitC} it was also observed that the Corson
compactness of \(\maProd{X,f,\alpha}\) was linked to
the countability of the ordinal \(\alpha\) (as
the same holds for \(\check\alpha=\alpha+1\)).
But the following generalization remains open.

\begin{question}
Let \(P\) be a partial (or linear) order. Is \(\check P\)
Corson compact if and only if \(P\) is countable?
\end{question}

If so, then for Corson compact \(X\) and \(f\)
satisfying condition \(\Gamma\), \(\maProd{X,f,P}\)
would be Corson compact if and only if \(P\) is
countable.

\section{Acknowledgements}

We thank an anonymous referee and Scott Varagona for
their comments that significantly improved this paper.

\bibliographystyle{unsrt}
\bibliography{sample}

\def\cprime{$'$}
\begin{thebibliography}{1}

\bibitem{M5}
Sina Greenwood and Judy Kennedy.
\newblock Connected generalized inverse limits over intervals.
\newblock {\em Fund. Math.}, 236(1):1--43, 2017.

\bibitem{Char}
W\l odzimierz~J. Charatonik and Robert~P. Roe.
\newblock On {M}ahavier products.
\newblock {\em Topology Appl.}, 166:92--97, 2014.

\bibitem{MR3014043}
W.~T. Ingram and William~S. Mahavier.
\newblock {\em Inverse limits: From continua to chaos}, volume~25 of {\em
  Developments in Mathematics}.
\newblock Springer, New York, 2012.

\bibitem{CLONTZVARAGONA}
Steven Clontz and Scott Varagona.
\newblock Destruction of metrizability in generalized inverse limits.
\newblock {\em Topology Proc.}, 48:289--297, 2016.

\bibitem{GitC}
Steven Clontz and Scott Varagona.
\newblock Mahavier products, idempotent relations, and {C}ondition {$\Gamma$}.
\newblock {\em Topology Proc.}, 54:259--269, 2019.

\bibitem{1954}
L.E Ward.
\newblock Partially ordered topological spaces.
\newblock {\em Proceedings of the American Mathematical Society}, 5:144--161,
  1954.

\bibitem{gierz_hofmann_keimel_lawson_mislove_scott_2003}
G.~Gierz, K.~H. Hofmann, K.~Keimel, J.~D. Lawson, M.~Mislove, and D.~S. Scott.
\newblock {\em Continuous Lattices and Domains}.
\newblock Encyclopedia of Mathematics and its Applications. Cambridge
  University Press, 2003.

\end{thebibliography}

\end{document}